\documentclass[12pt,fleqn,leqno]{amsart}
\usepackage{amssymb}
\usepackage{mathrsfs}
\usepackage[bb=boondox]{mathalfa} 
\usepackage{enumitem}

\usepackage[svgnames]{xcolor}
\usepackage[colorlinks,linkcolor=blue,citecolor=Green]{hyperref}
\usepackage{titlesec}
\usepackage{tikz-cd}
\usetikzlibrary{decorations.pathmorphing}

\titleformat{\section}[block]
 {\bfseries}
 {\thesection.}
 {\fontdimen2\font}
 {}

\setlist{noitemsep}
\setenumerate{labelindent=\parindent,label=\upshape{(\alph*)}}

\newtheorem{theorem}{Theorem}[section]
\newtheorem{corollary}[theorem]{Corollary}
\newtheorem{proposition}[theorem]{Proposition}
\newtheorem{lemma}[theorem]{Lemma}

\theoremstyle{definition}
\newtheorem{remark}[theorem]{Remark}
\newtheorem{example}[theorem]{Example}

\DeclareMathOperator{\uhr}{\upharpoonright} 
\DeclareMathOperator{\st}{st} 
\DeclareMathOperator\sto{\leadsto}
\DeclareMathOperator\card{card}
\newcommand{\embed}[1]{\stackrel{#1}{\hookrightarrow}}
\newcommand{\skobi}[1]{[#1]}

\renewcommand{\emptyset}{\varnothing}
\numberwithin{equation}{section}

\begin{document}

\author{Valentin Gutev}

\address{Department of Mathematics, Faculty of Science, University of
     Malta, Msida MSD 2080, Malta}

\email{valentin.gutev@um.edu.mt}



\title[Constructing Selections Stepwise Over Cones]{Constructing
  Selections Stepwise Over Cones of Simplicial Complexes}

\begin{abstract}
  It is obtained a natural generalisation of Uspenskij's selection
  characterisation of paracompact $C$-spaces. The method developed to
  achieve this result is also applied to give a simplified proof of a
  similar characterisation of paracompact finite $C$-space obtained
  previously by Valov. Another application is a characterisation of
  finite-dimensional paracompact spaces which generalises both a
  remark done by Michael and a result obtained by the author.
\end{abstract}

\date{\today}
\maketitle

\section{Introduction}

All spaces in this paper are Hausdorff topological spaces. A space $X$
has property $C$, or is a \emph{$C$-space}, \label{page-c-space}
if for any sequence $\{\mathscr{U} _n:n<\omega\}$ of open covers of
$X$ there exists a sequence $\{\mathscr{V} _n:n<\omega\}$ of open
pairwise-disjoint families in $X$ such that each $\mathscr{V} _n$
refines $\mathscr{U} _n$ and $\bigcup_{n<\omega}\mathscr{V} _n$ is a
cover of $X$.  The $C$-space property was originally defined by W.\
Haver \cite{haver:1974} for compact metric spaces, subsequently Addis
and Gresham \cite{addis-gresham:78} reformulated Haver's definition
for arbitrary spaces.  It should be remarked that a $C$-space $X$ is
paracompact if and only if it is countably paracompact and normal, see
e.g.\ \cite[Proposition 1.3]{MR2352366}. Every finite-dimensional
paracompact space, as well as every countable-dimensional metrizable
space, is a $C$-space \cite{addis-gresham:78}, but there exists a
compact metric $C$-space which is not countable-dimensional
\cite{pol:81}.\medskip

In what follows, we will use $\Phi:X\sto Y$ to designate that $\Phi$
is a map from $X$ to the nonempty subsets of $Y$, i.e.\ a
\emph{set-valued mapping}. A map $f:X\to Y$ is a \emph{selection} for
$\Phi:X\sto Y$ if $f(x)\in \Phi(x)$, for all $x\in X$. A mapping
$\Phi:X\sto Y$ is \emph{lower locally constant}, see \cite{gutev:05},
if the set $\{x\in X:K\subset \Phi(x)\}$ is open in $X$, for every
compact subset $K\subset Y$.  This property appeared in a paper of
Uspenskij \cite{uspenskij:98}; later on, it was used by some authors
(see, for instance, \cite{chigogidze-valov:00a,valov:00}) under the
name ``strongly l.s.c.'', while in papers of other authors strongly
l.s.c.\ was already used for a different property of set-valued
mappings (see, for instance, \cite{gutev:95e}). Regarding our
terminology, let us remark that a singleton-valued mapping (i.e.\ a
usual map) is lower locally constant precisely when it is locally
constant. \medskip

Finally, let us recall that a space $S$ is \emph{aspherical} if every
continuous map of the $k$-sphere ($k\geq 0$) in $S$ can be extended to
a continuous map of the $(k+1)$-ball in $S$.  The following theorem
was obtained by Uspenskij \cite[Theorem 1.3]{uspenskij:98}.

\begin{theorem}[\cite{uspenskij:98}]
  \label{theorem-Select-Ext-v6:1}
  A paracompact space $X$ is a $C$-space if and only if for every
  topological space $Y$, each lower locally constant aspherical-valued
  mapping $\Phi : X\sto Y$ has a continuous selection.
\end{theorem}

For $k\ge 0$ and subsets $S, B\subset Y$, we will write that
$S\embed{k} B$ if every continuous map of the $k$-sphere in $S$ can be
extended to a continuous map of the $(k+1)$-ball in $B$.  Similarly,
for mappings $\varphi, \psi :X\sto Y$, we will write
$\varphi \embed{k}\psi$ to express that $\varphi(x)\embed{k}\psi(x)$,
for every $x\in X$.  In these terms, we shall say that a sequence of
mappings $\varphi_n: X\sto Y$, $n<\omega$, is \emph{aspherical} if
$\varphi_n\embed{n}\varphi_{n+1}$, for every $n<\omega$. Also, to each
sequence of mappings $\varphi_n:X\sto Y$, $n<\omega$, we will
associate its union $\bigcup_{n<\omega}\varphi_n:X\sto Y$, defined
pointwise by
$\left[\bigcup_{n<\omega}\varphi_n\right]
(x)=\bigcup_{n<\omega}\varphi_n(x)$, $x\in X$.\medskip

In the present paper, we will show that the mapping $\Phi:X\sto Y$ in
Theorem~\ref{theorem-Select-Ext-v6:1} can be replaced by an aspherical
sequence of lower locally constant mappings $\varphi_n:X\sto Y$,
$n<\omega$. Namely, the following theorem will be proved.

\begin{theorem}
  \label{theorem-Select-Ext-v6:2}
  A paracompact space $X$ is a $C$-space if and only if for every
  topological space $Y$, each aspherical sequence $\varphi_n:X\sto Y$,
  $n<\omega$, of lower locally constant mappings admits a continuous
  selection for its union $\bigcup_{n<\omega}\varphi_n$.
\end{theorem}

By taking $\varphi_n=\varphi_{n+1}$, $n<\omega$, the selection
property in Theorem \ref{theorem-Select-Ext-v6:2} immedia\-tely
implies that of Theorem \ref{theorem-Select-Ext-v6:1}. Implicitly, the
selection property in Theorem~\ref{theorem-Select-Ext-v6:1} also
implies that of Theorem \ref{theorem-Select-Ext-v6:2} because both
these properties are equivalent to $X$ being a $C$-space.  However,
the author is not aware of any explicit argument show\-ing this. In this
regard and in contrast to Theorem \ref{theorem-Select-Ext-v6:1}, the
proof of Theorem~\ref{theorem-Select-Ext-v6:2} is straightforward
in both directions. Here is briefly the idea behind this
proof.\medskip 

In the next section, we deal with a simple construction of continuous
extensions of maps over cones of simplicial complexes, see Proposition
\ref{proposition-Select-Ext-v3:1}. This construction is applied in
Section \ref{sec:skeletal-selections} to special simplicial complexes
which are defined in Section \ref{sec:finite-asph-sequ}. Namely, in
Section \ref{sec:finite-asph-sequ}, to each cover
$\bigcup_{n<\omega} \mathscr{U}_n$ of $X$ consisting of families
$\mathscr{U}_n$, $n<\omega$, of subsets of $X$, we associate a
subcomplex $\Delta(\mathscr{U}_{<\omega})$ of the nerve
$\mathscr{N}(\mathscr{U}_{<\omega})$, where $\mathscr{U}_{<\omega}$
stands for the disjoint union $\bigsqcup_{n<\omega}\mathscr{U}_n$, see
Example~\ref{example-Select-Ext-v12:1}. Intuitively,
$\Delta(\mathscr{U}_{<\omega})$ consists of those simplices
$\sigma\in \mathscr{N}(\mathscr{U}_{<\omega})$ which have at most one
vertex in each $\mathscr{U}_n$, $n<\omega$. One benefit of this
subcomplex is that
$\Delta(\mathscr{V}_{<\omega})= \mathscr{N}(\mathscr{V}_{<\omega})$,
whenever the families $\mathscr{V}_n$, $n<\omega$, are as in the
defining property of $C$-spaces, i.e.\ pairwise-disjoint, see
Proposition \ref{proposition-Select-Ext-v16:3}. Another benefit is
that each sequence of covers $\mathscr{U}_n$, $n<\omega$, of $X$
generates a natural aspherical sequence of mappings on $X$. This is
done by considering the sequence of subcomplexes
$\Delta(\mathscr{U}_{\leq n})$, $n<\omega$, where each
$\Delta(\mathscr{U}_{\leq n})$ is defined as above with respect to the
indexed cover $\mathscr{U}_{\leq n}=\bigsqcup_{k=0}^n \mathscr{U}_k$
of $X$. Then each $\Delta(\mathscr{U}_{\leq n})$, $n<\omega$,
generates a simplicial-valued mapping
$\Delta_{[\mathscr{U}_{\leq n}]}:X\sto \Delta(\mathscr{U}_{\leq n})$
which assigns to each $x\in X$ the simplicial complex
$\Delta_{[\mathscr{U}_{\leq n}]}(x)$ of those simplices
$\sigma\in \Delta(\mathscr{U}_{\leq n})$ for which
$x\in \bigcap\sigma$. Finally, we consider the geometric realisation
$\left|\Delta(\mathscr{U}_{\leq n})\right|$ of
$\Delta(\mathscr{U}_{\leq n})$ and the set-valued mapping
$\left|\Delta_{[\mathscr{U}_{\leq n}]}\right|:X\sto
\left|\Delta(\mathscr{U}_{\leq n})\right|\subset
\left|\Delta(\mathscr{U}_{<\omega})\right|$ corresponding to the
``geometric realisation'' of $\Delta_{[\mathscr{U }_{\leq n}]}$.
Thus, for open locally finite covers $\mathscr{U}_n$, ${n<\omega}$, of
$X$, the sequence
$\left|\Delta_{[\mathscr{U}_{\leq n}]}\right|:X\sto
\left|\Delta(\mathscr{U}_{<\omega})\right|$, $n<\omega$, is always
aspherical and consists of lower locally constant mappings,
Propositions \ref{proposition-Select-Ext-v16:1} and
\ref{proposition-Select-Ext-v16:2}. In Section
\ref{sec:skeletal-selections}, by restricting $X$ to be a paracompact
space, we show that each aspherical sequence $\varphi_n:X\sto Y$,
$n<\omega$, of lower locally constant mappings admits a sequence
$\mathscr{F}_n$, $n<\omega$, of closed locally finite interior covers
of $X$ and a continuous map
$f:\left|\Delta(\mathscr{F}_{<\omega})\right|\to Y$ such that each
composite mapping
${f\circ \left|\Delta_{[\mathscr{F}_{\leq n}]}\right|:X\sto Y}$ is a
set-valued selection for $\varphi_n$, $n<\omega$, see Theorem
\ref{theorem-Select-Ext-v14:1}. This is applied in
Section~\ref{sec:select-canon-maps} to show that the selection problem
in Theorem \ref{theorem-Select-Ext-v6:2} is now equivalent to that of
the mappings
$\left|\Delta_{[\mathscr{U}_{\leq n}]}\right|:X\sto
\left|\Delta(\mathscr{U}_{<\omega})\right|$, $n<\omega$, corresponding
to open locally finite covers $\mathscr{U}_n$, $n<\omega$, of $X$,
Theorem \ref{theorem-Select-Ext-vgg-rev:1}. In the same section, this
selection problem is further reduced to the existence of canonical
maps $f:X\to \left|\Delta(\mathscr{U}_{<\omega})\right|$, Corollary
\ref{corollary-Select-Ext-v22:1}.  Finally, in Section
\ref{sec:canon-maps-sequ}, it is shown that the existence of canonical
maps $f:X\to \left|\Delta(\mathscr{U}_{<\omega})\right|$ is equivalent
to $C$-like properties of $X$, Theorem
\ref{theorem-Select-Ext-v18:2}. Theorem \ref{theorem-Select-Ext-v6:2}
is then obtained as a special case of Theorem
\ref{theorem-Select-Ext-v18:2}, see Corollary
\ref{corollary-Select-Ext-v18:1}. Here, let us explicitly remark that
two other special cases of Theorem \ref{theorem-Select-Ext-v18:2} are
covering two other similar results --- Corollary
\ref{corollary-Select-Ext-v18:2} (a selection theorem of Valov
\cite[Theorem 1.1]{valov:00} about finite $C$-spaces) and Corollary
\ref{corollary-Select-Ext-v18:3} (generalising both a remark done by
Michael in \cite[Remark 2]{uspenskij:98} and a result obtained by the
author in \cite[Theorem~3.1]{gutev:2018a}).

\section{Extensions of maps over cones of simplicial complexes}
\label{sec:extensions-maps-over}

By a \emph{simplicial complex} we mean a collection $\Sigma$ of
nonempty finite subsets of a set $S$ such that $\tau\in \Sigma$,
whenever $\emptyset\neq \tau\subset \sigma\in \Sigma$. The set
$\bigcup\Sigma $ is the \emph{vertex set} of $\Sigma$, while each
element of $\Sigma$ is called a \emph{simplex}. The
\emph{$k$-skeleton} $\Sigma^k$ of $\Sigma$ ($k\geq 0$) is the
simplicial complex
$\Sigma^{k}=\{\sigma\in \Sigma:\card(\sigma)\leq k+1\}$, where
$\card(\sigma)$ is the cardinality of $\sigma$. In the sequel, for
simplicity, we will identify the vertex set of $\Sigma$ with its
$0$-skeleton $\Sigma^0$, and will say that $\Sigma$ is a
\emph{$k$-dimensional} if $\Sigma=\Sigma^k$.\medskip

The vertex set $\Sigma^0$ of each simplicial complex $\Sigma$ can be
embedded as a linearly independent subset of some linear normed
space. Then to any simplex $\sigma\in \Sigma$, we may associate the
corresponding \emph{geometric simplex} $|\sigma|$ which is the
convex hull of $\sigma$. Thus, $\card(\sigma)=k+1$ if and only if
$|\sigma|$ is a \emph{$k$-dimensional simplex}. Finally, we set
$|\Sigma|=\bigcup\{|\sigma|:\sigma\in \Sigma\}$ which is called the
\emph{geometric realisation} of $\Sigma$. As a topological space, we
will always consider $|\Sigma|$ endowed with the \emph{Whitehead
  topology} \cite{MR1576810,MR0030759}. This is the topology in which
a subset $U\subset |\Sigma|$ is open if and only if $U\cap |\sigma|$
is open in $|\sigma|$, for every $\sigma\in \Sigma$. \medskip

The \emph{cone} $Z*v$ over a space $Z$ with a vertex $v$ is the
quotient space of $Z\times[0,1]$ obtained by identifying all points of
$Z\times\{1\}$ into a single point $v$.  For a simplicial complex
$\Sigma$ and a point $v$ with $v\notin \Sigma^0$, the \emph{cone} on
$\Sigma$ with a vertex $v$ is the simplicial complex defined by
\[
  \Sigma*v=\Sigma\cup\left\{\sigma\cup \{v\}:
    \sigma\in\Sigma\right\}\cup \{\{v\}\}.
\]
Evidently, we have that $|\Sigma|*v=|\Sigma*v|$.

\begin{proposition}
  \label{proposition-Select-Ext-v3:1}
  Let ${S_0,\dots,S_{n+1}\subset Y}$ with
  ${\emptyset\neq S_0\embed{0} S_1\embed{1} \cdots \embed{n}
    S_{n+1}}$, and $g:|\Sigma|\to S_n$ be a continuous map from an
  $n$-dimensional simplicial complex $\Sigma$ such that
  $g(|\Sigma^k|)\subset S_k$, for every $k\leq n$. If
  $v\notin \Sigma^0$, then $g$ can be extended to a continuous map
  $h:|\Sigma*v|\to S_{n+1}$ such that $h(|(\Sigma*v)^k|)\subset S_k$,
  for every $k\leq n+1$.
\end{proposition}

\begin{proof}
  By a finite induction, extend each restriction
  $g_k=g\uhr\left|\Sigma^k\right|$ to a continuous map
  ${h_k:\left|(\Sigma*v)^k\right|\to S_k}$, $k\leq n$, such that
  $h_k\uhr \left|(\Sigma*v)^{k-1}\right|=h_{k-1}$, $k>0$. Briefly,
  define $h_0:(\Sigma*v)^0\to S_0$ by $h_0\uhr \Sigma^0= g_0$ and
  $h_0(v)\in S_0$. Whenever $u\in \Sigma^0$ is a vertex of $\Sigma$,
  the map $h_0:\{u,v\}\to S_0$ can be extended to a continuous map
  $h_{(1,u)}:|\{u,v\}|\to S_1$ because $S_0\embed{0} S_1$. Then the
  map $h_1:\left|(\Sigma*v)^1\right|\to S_1$ defined by
  $h_1\uhr (\Sigma*v)^0=h_0$, $h_1\uhr |\Sigma^1|=g_1$ and
  $h_1\uhr |\{u,v\}|=h_{(1,u)}$, $u\in \Sigma^0$, is a continuous
  extension of both $h_0$ and $g_1$. The construction can be carried
  on by induction to get a continuous extension
  $h_n:|(\Sigma*v)^n|\to S_n$ of $g=g_n$ with the required
  properties. Finally, if $\sigma\in \Sigma*v$ is an
  $(n+1)$-dimensional simplex, then $h_n$ is defined on the boundary
  $|\sigma|\cap |(\Sigma*v)^n|$ of $|\sigma|$ which is homeomorphic to
  the $n$-sphere. Hence, it can be extended to a continuous map
  $h_\sigma:|\sigma|\to S_{n+1}$ because $S_n\embed{n} S_{n+1}$. The
  required map $h:|\Sigma*v|\to S_{n+1}$ is now defined by
  $h\uhr |(\Sigma*v)^n|=h_n$ and $h\uhr |\sigma|=h_\sigma$, for every
  $(n+1)$-dimensional simplex $\sigma\in \Sigma*v$.
\end{proof}

\section{Nerves of sequences of covers}
\label{sec:finite-asph-sequ}

The set $\Sigma_S$ of all nonempty finite subsets of a set $S$ is a
simplicial complex. Another natural example is the \emph{nerve} of an
indexed cover $\{U_\alpha:\alpha\in \mathscr{A}\}$ of a
set $X$, which is the subcomplex of $\Sigma_\mathscr{A}$ defined by
\begin{equation}
  \label{eq:Select-Ext-v1:1}
  \mathscr{N}(\mathscr{A})= \left\{\sigma\in
    \Sigma_\mathscr{A}:\bigcap_{\alpha\in
      \sigma}U_\alpha\neq\emptyset\right\}.   
\end{equation}
Following Lefschetz \cite{MR0007093}, the intersection
$\bigcap_{\alpha\in \sigma}U_\alpha$ is called the \emph{kernel} of
$\sigma$, and is often denoted by $\ker[\sigma]=\bigcap_{\alpha\in
  \sigma}U_\alpha$. In case $\mathscr{U}$ is an unindexed cover of
$X$, its nerve is denoted by $\mathscr{N}(\mathscr{U})$. In this case,
$\mathscr{U}$ is indexed by itself, and each simplex $\sigma\in
\mathscr{N}(\mathscr{U})$ is merely a nonempty finite subset of
$\mathscr{U}$ with $\ker[\sigma]=\bigcap \sigma\neq
\emptyset$.\medskip

Here, an important role will be played by a subcomplex of the nerve of
a special indexed cover of $X$. The prototype of this subcomplex can be
found in some of the considerations in the proof of
\cite[Theorem 2.1]{uspenskij:98}. 

\begin{example}
  \label{example-Select-Ext-v12:1}
  Whenever $0<\kappa\leq \omega$, let $\mathscr{U}_n$, $n<\kappa$, be
  families of subsets of $X$ such that
  $\bigcup_{n<\kappa}\mathscr{U}_n$ is a cover of $X$, and let 
  $\bigsqcup_{n<\kappa}\mathscr{U}_n$ be the \emph{disjoint union} of
  these families (obtained, for instance, by identifying each
  $\mathscr{U}_n$ with $\mathscr{U}_n\times\{n\}$, $n<\kappa$). The
  nerve of this indexed cover of $X$ defines a natural simplicial
  complex
  \begin{equation}
    \label{eq:Select-Ext-v12:2}
    \mathscr{N}\left(\mathscr{U}_{<\kappa}\right)=
    \mathscr{N}\left(\bigsqcup_{n<\kappa}\mathscr{U}_n\right).  
  \end{equation}
  A simplex $\sigma\in \mathscr{N}(\mathscr{U}_{<\kappa})$ can be
  described as the disjoint union $\sigma=\bigsqcup_{i=1}^m \sigma_i$
  of finitely many simplices
  $\sigma_i\in \mathscr{N}(\mathscr{U}_{n_i})$,
  for $n_1<\dots <n_m<\kappa$, such that
  $\bigcap_{i=1}^m \ker[\sigma_i]\neq \emptyset$. The simplicial
  complex $\mathscr{N}(\mathscr{U}_{<\kappa})$ contains a natural
  subcomplex $\Delta(\mathscr{U}_{<\kappa})$, define by
  \begin{equation}
    \label{eq:Select-Ext-v12:4}
    \Delta(\mathscr{U}_{<\kappa})=\big\{\sigma\in
      \mathscr{N}(\mathscr{U}_{<\kappa}): 
      \card(\sigma\cap \mathscr{U}_n\times\{n\})\leq 1,\
      n<\kappa\big\}. 
  \end{equation}
  In other words, the subcomplex $\Delta(\mathscr{U}_{<\kappa})$
  consists of those simplices
  $\sigma\in \mathscr{N}\left(\mathscr{U}_{<\kappa}\right)$ which are
  composed of finitely many vertices
  $U_i=(U_{i},n_i)\in \mathscr{U}_{n_i}$, $i\leq m$, where
  $n_1<\dots <n_m<\kappa$. In the special case of $\kappa=n+1<\omega$,
  we will simply write
  $\mathscr{N}(\mathscr{U}_{\leq
    n})=\mathscr{N}(\mathscr{U}_{<\kappa})$ and
  $\Delta(\mathscr{U}_{\leq n})=\Delta(\mathscr{U}_{<\kappa})$.\qed
\end{example}
  
The subcomplex $\Delta(\mathscr{U}_{<\kappa})$ in Example
\ref{example-Select-Ext-v12:1} is naturally related to the definition
of $C$-spaces. The following proposition is an immediate consequence
of \eqref{eq:Select-Ext-v12:4}.

\begin{proposition}
  \label{proposition-Select-Ext-v16:3}
  Let $0<\kappa\leq \omega$ and $\mathscr{V}_n$, $n<\kappa$, be a
  sequence of pairwise-disjoint families of subsets of $X$, whose
  union forms a cover of $X$.  Then
  \begin{equation}
    \label{eq:Select-Ext-v12:3}
    \Delta(\mathscr{V}_{<\kappa})=\mathscr{N}(\mathscr{V}_{<\kappa}).
  \end{equation}
\end{proposition}

For a simplicial complex $\Sigma$, a mapping $\Omega:X\sto \Sigma$
will be called \emph{simplicial-valued} if $\Omega(p)$ is a subcomplex
of $\Sigma$, for each $p\in X$. Such a mapping $\Omega:X\sto \Sigma$
generates a mapping $|\Omega|:X\sto |\Sigma|$ defined by
\begin{equation}
  \label{eq:Select-Ext-v16:2}
  |\Omega|(p)=\left|\Omega(p)\right|= \bigcup_{\sigma\in
    \Omega(p)}|\sigma|,\quad p\in X. 
\end{equation}
Here is a natural example.  Each indexed cover
$\{U_\alpha:\alpha\in \mathscr{A}\}$ of $X$ generates a natural
simplicial-valued mapping
$\Sigma_\mathscr{A}:X\sto \Sigma_\mathscr{A}$, defined by
\begin{equation}
  \label{eq:Select-Ext-v15:2}
  \Sigma_\mathscr{A}(p)=\left\{\sigma\in \Sigma_\mathscr{A}:
    p\in\bigcap_{\alpha\in
      \sigma}U_\alpha\right\}, \quad p\in X.  
\end{equation}
In fact, each $\Sigma_\mathscr{A}(p)$ is a subcomplex of
$\mathscr{N}(\mathscr{A})$, so
$\Sigma_\mathscr{A}:X\sto \mathscr{N}(\mathscr{A})$. \medskip

The benefit of the mapping in \eqref{eq:Select-Ext-v15:2} comes in the
setting of the simplicial complex $\Delta(\mathscr{U}_{<\kappa})$ in
Example \ref{example-Select-Ext-v12:1} associated to a sequence of
covers $\mathscr{U}_n$, $n<\kappa$, of $X$ for some
$0<\kappa\leq \omega$. Namely, we may define the corresponding
simplicial-valued mapping
$\Delta_{[\mathscr{U}_{<\kappa}]}:X\sto \Delta(\mathscr{U}_{<\kappa})$
by the same pattern as in \eqref{eq:Select-Ext-v15:2}, i.e.
\begin{equation}
  \label{eq:Select-Ext-v15:3}
  \Delta_{[\mathscr{U}_{<\kappa}]}(p)= \big\{\sigma\in
    \Delta(\mathscr{U}_{<\kappa}): 
    p\in\ker[\sigma]\big\}, \quad p\in X.
\end{equation}
Just like before, we will write
$\Delta_{[\mathscr{U}_{\leq n}]}=\Delta_{[\mathscr{U}_{<\kappa}]}$
whenever $\kappa=n+1<\omega$. \medskip

We now have the following natural relationship with aspherical
sequences of lower locally constant mappings.

\begin{proposition}
  \label{proposition-Select-Ext-v16:1}
  Let $\mathscr{U}_n$, $n<\omega$, be a sequence of covers of a set
  $X$. Then 
  \begin{equation}
    \label{eq:Select-Ext-v16:1}
    \Delta_{[\mathscr{U}_{\leq n}]}(p)*U\subset \Delta_{[\mathscr{U}_{\leq
      n+1}]}(p), \quad\text{whenever $p\in U\in \mathscr{U}_{n+1}$ and
    $n<\omega$.}
  \end{equation}
  Accordingly,
  $\left|\Delta_{[\mathscr{U}_{\leq n}]}\right|:X\sto
  |\Delta(\mathscr{U}_{\leq n})|\subset
  |\Delta(\mathscr{U}_{<\omega})|$, $n<\omega$, is an aspherical
  sequence of mappings.
\end{proposition}

\begin{proof}
  The property in \eqref{eq:Select-Ext-v16:1} follows from the fact
  that $U=(U,n+1)\notin \bigsqcup_{k=0}^n\mathscr{U}_k$, whenever 
  $p\in U\in \mathscr{U}_{n+1}$. Since
  $\left|\Delta_{[\mathscr{U}_{\leq
        n}]}(p)*U\right|=\left|\Delta_{[\mathscr{U}_{\leq
        n}]}(p)\right|*U$ is contractible, this implies that
  $\left|\Delta_{[\mathscr{U}_{\leq n}]}(p)\right|\embed{n} \left|
      \Delta_{[\mathscr{U}_{\leq n}]}(p)*U\right| \subset
    \left|\Delta_{[\mathscr{U}_{\leq n+1}]}(p)\right|$.  
\end{proof}

\begin{proposition}
  \label{proposition-Select-Ext-v16:2}
  Let $\mathscr{U}_0,\dots, \mathscr{U}_n$ be a sequence of
  point-finite open covers of a space $X$. Then the mapping
  $\left|\Delta_{[\mathscr{U}_{\leq n}]}\right|:X\sto
  |\Delta(\mathscr{U}_{\leq n})|$ is lower locally constant.
\end{proposition}

\begin{proof}
  Whenever $p\in X$, the set
  $V_p=\bigcap\left\{U\in \bigcup_{k\leq n}\mathscr{U}_k:
    p\in U\right\}$ is a neighbourhood of $p$. Take  a point
  $q\in V_p$. Then by \eqref{eq:Select-Ext-v15:3}, 
  $\sigma\in \Delta_{[\mathscr{U}_{\leq n}]}(p)$ implies that
  $\sigma\in \Delta_{[\mathscr{U}_{\leq n}]}(q)$ because $q\in
  V_p\subset \ker[\sigma]$. Thus, 
  $\Delta_{[\mathscr{U}_{\leq n}]}(p)\subset
  \Delta_{[\mathscr{U}_{\leq n}]}(q)$ and, accordingly,
  $\left|\Delta_{[\mathscr{U}_{\leq n}]}\right|$ is lower locally
  constant.
\end{proof}

We conclude this section with a remark about the importance of
disjoint unions in the definition of the subcomplex
$\Delta(\mathscr{U}_{<\kappa})$ in Example
\ref{example-Select-Ext-v12:1}.

\begin{remark}
  \label{remark-Select-Ext-v19:1}
  For a sequence of covers $\mathscr{U}_n$, $n<\kappa$, of $X$, where
  $0<\kappa\leq \omega$, one can define the subcomplex
  $\Delta(\mathscr{U}_{<\kappa})\subset
  \mathscr{N}(\mathscr{U}_{<\kappa})$ by considering
  $\mathscr{N}(\mathscr{U}_{<\kappa})$ to be the nerve of the usual
  unindexed cover $\bigcup_{n<\kappa}\mathscr{U}_n$, rather than the
  disjoint union $\bigsqcup_{n<\kappa}\mathscr{U}_n$. However, this
  will not work to establish a property similar to that in Proposition
  \ref{proposition-Select-Ext-v16:1}, also for the essential results
  in the next sections (see, for instance, Theorem
  \ref{theorem-Select-Ext-v14:1} and Lemma
  \ref{lemma-Select-Ext-v4:1}). Namely, suppose that $\mathscr{U}_0$
  and $\mathscr{U}_1$ are covers of $X$ which contain elements
  $U_i\in \mathscr{U}_i$, $i=0,1$, with $U_0\cap U_1\neq \emptyset$
  and $U_i\notin \mathscr{U}_{1-i}$. Then
  $\sigma=\{U_0,U_1\}\in \Delta(\mathscr{U}_{\leq 1})$. However, if
  $\mathscr{U}_2$ is a cover of $X$ with $U_0,U_1\in \mathscr{U}_2$,
  and $\Delta(\mathscr{U}_{\leq 2})$ is defined on the basis of
  unindexed covers, then
  $\Delta(\mathscr{U}_{\leq 1})\not\subset \Delta(\mathscr{U}_{\leq
    2})$ because
  $\sigma=\{U_0,U_1\}\notin \Delta(\mathscr{U}_{\leq
    2})$. \qed
\end{remark}

\section{Skeletal selections}
\label{sec:skeletal-selections}

For mappings $\varphi,\psi:X\sto Y$, we will write
$\varphi\subset \psi$ to express that $\varphi(p)\subset \psi(p)$, for
every $p\in X$. In this case, the mapping $\varphi$ is called a
\emph{set-valued selection}, or a \emph{multi-selection}, for
$\psi$. Also, let us recall that a cover $\mathscr{F}$ of a space $X$
is called \emph{interior} if the collection of the interiors of the
elements of $\mathscr{F}$ is a cover of $X$.\medskip

The following theorem will be proved in this section.

\begin{theorem}
  \label{theorem-Select-Ext-v14:1}
  Let $X$ be a paracompact space and $\varphi_n:X\sto Y$, $n<\omega$,
  be an aspherical sequence of lower locally constant mappings in a
  space $Y$. Then there exists a sequence $\mathscr{F}_n$, $n<\omega$,
  of closed locally finite interior covers of $X$ and a continuous map
  ${f:\big|\Delta(\mathscr{F}_{<\omega})\big|\to Y}$ such that
  \begin{equation}
    \label{eq:Select-Ext-v14:1}
    f\circ \left|\Delta_{[\mathscr{F}_{\leq n}]}\right|\subset
    \varphi_n,\quad \text{for every $n<\omega$.} 
  \end{equation}
\end{theorem}

Let us explicitly remark that, here,
$\Delta_{[\mathscr{F}_{\leq n}]}:X\sto \Delta(\mathscr{F}_{\leq
  n})\subset \Delta(\mathscr{F}_{<\omega})$
is the simplicial-valued mapping associated to the covers
$\mathscr{F}_k$, $k\leq n$, see \eqref{eq:Select-Ext-v15:3}, while
$f\circ \left|\Delta_{[\mathscr{F}_{\leq n}]}\right|$ is the composite
mapping
\begin{center}
  \begin{tikzcd}
    &&\lvert\Delta(\mathscr{F}_{<\omega})\rvert \arrow[d, "f"]\\
    {X} \arrow[urr, rightsquigarrow,
    "\left\lvert\Delta_{\skobi{\mathscr{F}_{\leq n}}}\right\rvert", bend left=15]
    \arrow[rr, rightsquigarrow, "f\circ
    \left\lvert\Delta_{\skobi{\mathscr{F}_{\leq n}}} \right\rvert"] &&
    Y
  \end{tikzcd}
\end{center}
According to the definition of
$\Delta_{[\mathscr{F}_{\leq n}]}:X\sto \Delta(\mathscr{F}_{\leq n})$,
see also \eqref{eq:Select-Ext-v16:2}, the property in
\eqref{eq:Select-Ext-v14:1} means that
$f(|\sigma|) \subset \varphi_n(p)$, for every
$\sigma \in \Delta(\mathscr{F}_{\leq n})$ and
$p\in\ker[\sigma]$.\medskip

Turning to the proof of Theorem \ref{theorem-Select-Ext-v14:1}, let us
observe that the simplicial complex $\Delta(\mathscr{F}_{\leq n})$ is
$n$-dimensional, see \eqref{eq:Select-Ext-v12:4} of Example
\ref{example-Select-Ext-v12:1}. In what follows, its $k$-skeleton will
be denoted by $\Delta^k(\mathscr{F}_{\leq n})$. In these terms,
following the idea of an $n$-skeletal selection in \cite{gutev:2018a},
we shall say that a continuous map
$f:|\Delta(\mathscr{F}_{\leq n})|\to Y$ is a \emph{skeletal selection}
for a sequence of mappings $\varphi_0,\dots,\varphi_n:X\sto Y$ if
\begin{equation}
  \label{eq:Select-Ext-v4:3}
  f(|\sigma|) \subset \varphi_k(p),\ \text{for every $\sigma \in
    \Delta^k(\mathscr{F}_{\leq n})$, 
    $k\leq n$, and $p\in\ker[\sigma]$.} 
\end{equation}
Precisely as in \eqref{eq:Select-Ext-v15:3}, for each $k\leq n$ we may
associate the simplicial-valued mapping
$\Delta^k_{[\mathscr{F}_{\leq n}]}:X\sto \Delta^k(\mathscr{F}_{\leq
  n})$, which assigns to each $p$ in $ X$ the $k$-skeleton
$\Delta^k_{[\mathscr{F}_{\leq n}]}(p)$ of the subcomplex
$\Delta_{[\mathscr{F}_{\leq n}]}(p)\subset \Delta(\mathscr{F}_{\leq
  n})$. Then the property in \eqref{eq:Select-Ext-v4:3} means that the
composite mapping
$f\circ \left|\Delta^k_{[\mathscr{F}_{\leq n}]}\right|:X\sto
\left|\Delta^k(\mathscr{F}_{\leq n})\right|$ is a set-valued selection
for $\varphi_k$, for every $k\leq n$.\medskip

Finally, let us recall that a \emph{simplicial map}
$g:\Sigma_1\to \Sigma_2$ is a map $g:\Sigma_1^0\to \Sigma_2^0$ between
the vertices of simplicial complexes $\Sigma_1$ and $\Sigma_2$ such
that $g(\sigma)\in \Sigma_2$, for each $\sigma\in \Sigma_1$. If such a
map $g:\Sigma_1^0\to \Sigma_2^0$ is bijective, then the inverse
$g^{-1}$ is also a simplicial map, and we say that $g$ is a
\emph{simplicial isomorphism}. If $g$ is only injective, then $g$
embeds $\Sigma_1$ into $\Sigma_2$, so that we may consider $\Sigma_1$
as a subcomplex of $\Sigma_2$. Each simplicial map
$g:\Sigma_1\to \Sigma_2$ generates a continuous map
$|g|:|\Sigma_1|\to |\Sigma_2|$ which is affine on each geometric
simplex $|\sigma|$, for $\sigma\in \Sigma_1$.

\begin{lemma}
  \label{lemma-Select-Ext-v4:1}
  Let $Y$ be a space, $\mathscr{F}_0,\dots,\mathscr{F}_{n}$ be a
  sequence of closed locally finite covers of a paracompact space $X$,
  and $\varphi_0,\dots, \varphi_{n+1}:X\sto Y$ be a sequence of lower
  locally constant mappings with $\varphi_k \embed{k}\varphi_{k+1}$
  for every $k\leq n$.  If ${f_n:|\Delta(\mathscr{F}_{\leq n})|\to Y}$
  is a skeletal selection for $\varphi_0,\dots,\varphi_n$, then there
  exists a closed locally finite interior cover $\mathscr{F}_{n+1}$ of
  $X$ and a continuous extension
  $f_{n+1}:|\Delta(\mathscr{F}_{\leq n+1})|\to Y$ of $f_n$ which is a
  skeletal selection for $\varphi_0,\dots,\varphi_{n+1}$.
\end{lemma}

\begin{proof}
  Let
  $\Delta_{[\mathscr{F}_{\leq n}]}:X\sto \Delta(\mathscr{F}_{\leq n})$
  be the associated simplicial-valued mapping, defined as in
  \eqref{eq:Select-Ext-v15:3}. Whenever $p\in X$, the subcomplex
  \begin{equation}
    \label{eq:Select-Ext-v16:3}
    \Delta_p=\Delta_{[\mathscr{F}_{\leq n}]}(p)
  \end{equation}
  is $n$-dimensional such that, by \eqref{eq:Select-Ext-v4:3},
  $f_n\left(\left|\Delta^k_p\right|\right)\subset \varphi_k(p)$,
  $k\leq n$. Moreover, by hy\-pothesis,
  $\varphi_k(p) \embed{k}\varphi_{k+1}(p)$ for every $k\leq n$.  Since
  $p\notin \Delta_p^0$, it follows from Proposition
  \ref{proposition-Select-Ext-v3:1} that $f_n\uhr |\Delta_p|$ can be
  extended to a continuous map ${f_p:|\Delta_p* p|\to Y}$ such that
  \begin{equation}
    \label{eq:Select-Ext-v2:2}
    f_p\left(\left|(\Delta_p*
        p)^k\right|\right)\subset 
    \varphi_k(p),\quad \text{for every 
      $0\leq k\leq n+1$.}
  \end{equation}

  Since all covers are locally finite and closed, the point
  $p\in X$ is contained in the open set 
  \begin{equation}
    \label{eq:Select-Ext-v3:2}
    O_p=X\setminus \bigcup\left\{F\in \mathscr{F}_0\cup\dots
      \cup\mathscr{F}_n: p\notin F\right\}.
  \end{equation}
  For the same reason, $\Delta_p*p$ is a finite simplicial
  complex. Accordingly, each set
  $f_p\left(\left|(\Delta_p*p)^k\right|\right)$, $k\leq n+1$, is
  compact. Hence, by \eqref{eq:Select-Ext-v2:2} and the hypothesis
  that each mapping $\varphi_k$, $k\leq n+1$, is lower locally
  constant, we may shrink $O_p$ to a neighbourhood $V_p$ of $p$,
  defined by
  \begin{equation}
    \label{eq:Select-Ext-v1:3}
    V_{p}= \left\{x\in O_p:
      f_p\left(\left|(\Delta_p*p)^k\right|\right)\subset 
      \varphi_k(x),\ \text{for every $k\leq n+1$}\right\}.
  \end{equation}
  Finally, since $X$ is paracompact, it has an open locally finite
  cover $\mathscr{U}_{n+1}$ such that $\{V_{p}: p\in X\}$ is refined
  by the associated cover
  $\mathscr{F}_{n+1}=\left\{\overline{U}: U\in
    \mathscr{U}_{n+1}\right\}$ of the closures of the elements of
  $\mathscr{U}_{n+1}$. So, there is a map $p:\mathscr{F}_{n+1}\to X$
  such that
  \begin{equation}
    \label{eq:st-app-vgg-rev:4}
    F\subset V_{p(F)}\subset O_{p(F)},\quad \text{for every $F\in
      \mathscr{F}_{n+1}$.} 
  \end{equation}

  Having already defined the cover $\mathscr{F}_{n+1}$, we are going
  to extend $f_n$ to a skeletal selection
  $f_{n+1}:\big|\Delta(\mathscr{F}_{\leq n+1})\big|\to Y$ for the
  sequence $\varphi_0,\dots,\varphi_{n+1}$. To this end, take an
  $F\in \mathscr{F}_{n+1}$, and define the set
  \[
    \Delta_F=\big\{\tau\in \Delta(\mathscr{F}_{\leq n}):
      \tau\cup\{F\}\in \Delta(\mathscr{F}_{\leq n+1})\big\}.
  \]
  It is evident that $\Delta_F$ is a subcomplex of
  $\Delta(\mathscr{F}_{\leq n})$ with $F\notin \Delta_F^0$, hence the
  cone $\Delta_F*F$ is a subcomplex of
  $\Delta(\mathscr{F}_{\leq n+1})$.  Thus, to extend $f_n$ to a
  skeletal selection
  $f_{n+1}:\big|\Delta(\mathscr{F}_{\leq n+1})\big|\to Y$ for the
  sequence $\varphi_0,\dots,\varphi_{n+1}$, it now suffices to extend
  each $f_n\uhr |\Delta_F|$, $F\in \mathscr{F}_{n+1}$, to a continuous
  map $f_F:|\Delta_F*F|\to Y$ satisfying the condition in
  \eqref{eq:Select-Ext-v4:3} with respect to the simplices of
  $\Delta_F*F$. To this end, let us observe that
  \begin{equation}
    \label{eq:Select-Ext-v3:3}
    \Delta_F\subset \Delta_{p(F)}=\Delta_{[\mathscr{F}_{\leq n}]}(p(F)).
  \end{equation}
  Indeed, for $T\in\tau\in \Delta_F$, we have that
  $\emptyset\neq T\cap F\subset T\cap O_{p(F)}$, see
  \eqref{eq:st-app-vgg-rev:4}. Hence, by \eqref{eq:Select-Ext-v3:2},
  $p(F)\in T$ and according to \eqref{eq:Select-Ext-v15:3} and
  \eqref{eq:Select-Ext-v16:3}, $\tau\in \Delta_{p(F)}$.\smallskip

  We are now ready to define the required maps
  $f_F:|\Delta_F*F|\to Y$, $F\in \mathscr{F}_{n+1}$. Namely, by
  \eqref{eq:Select-Ext-v3:3}, we can embed $\Delta_{F}*F$ into the
  cone $\Delta_{p(F)}*p(F)$ by identifying $p(F)$ with $F$.  Let
  $\ell:\Delta_F*F\to \Delta_{p(F)}*p(F)$ be the corresponding
  simplicial embedding defined by $\ell\uhr \Delta^0_F$ to be the
  identity of $\Delta^0_F$, and $\ell(F)=p(F)$. Next, define a
  continuous extension $f_F:|\Delta_F*F|\to Y$ of $f_n\uhr |\Delta_F|$
  by $f_F=f_{p(F)}\circ |\ell|$. Take a simplex
  $\sigma\in (\Delta_F*F)^k$ for some $k\leq n+1$, and a point
  $x\in \ker[ \sigma]$. If $\sigma\in \Delta_F$, by the properties of
  $f_n$, see \eqref{eq:Select-Ext-v4:3},
  $f_F(|\sigma|)=f_n(|\sigma|)\subset \varphi_k(x)$. If $F\in \sigma$,
  then $x\in F\subset V_{p(F)}$ and, by \eqref{eq:Select-Ext-v1:3}, we
  have again that
  $f_F(|\sigma|)= f_{p(F)}(|\ell|(|\sigma|)) \subset
  \varphi_k(x)$. The proof is complete.
\end{proof}

Complementary to Lemma \ref{lemma-Select-Ext-v4:1} is the following
well-known property, see the proof of \cite[Theorem 2.1]{uspenskij:98}
and that of \cite[Theorem 3.1]{gutev:05}. The property itself was
stated explicitly in \cite[Proposition 3.2]{gutev:2018a}, and is an
immediate consequence of the definition of lower locally constant
mappings.

\begin{proposition}
  \label{proposition-Select-Ext-v4:1}
  If $X$ is a paracompact space and $\varphi:X\sto Y$ is a lower
  locally constant mapping, then there exists a closed locally finite
  interior cover $\mathscr{F}$ of $X$ and a
  \textup{(}continuous\textup{)} map
  $f:\Delta(\mathscr{F})=\mathscr{F}\to Y$ such that
  $f(F)\in \varphi(x)$, for every $x\in F\in \mathscr{F}$.
\end{proposition}

\begin{proof}[Proof of Theorem \ref{theorem-Select-Ext-v14:1}]
  Inductively, using Proposition \ref{proposition-Select-Ext-v4:1} and
  Lemma \ref{lemma-Select-Ext-v4:1}, there exists a sequence
  $\mathscr{F}_n$, $n<\omega$, of closed locally finite interior
  covers of $X$ and continuous maps
  $f_n:\big|\Delta(\mathscr{F}_{\leq n})\big|\to Y$, $n<\omega$, such
  that each $f_n$ is a skeletal selection for the sequence
  $\varphi_0,\dots,\varphi_n$, and each $f_{n+1}$ is an extension of
  $f_n$. Since
  $\Delta(\mathscr{F}_{<\omega})=\bigcup_{n<\omega}
  \Delta(\mathscr{F}_{\leq n})$, we may define a map
  $f:\big|\Delta(\mathscr{F}_{<\omega})\big|\to Y$ by
  $f\uhr \big|\Delta(\mathscr{F}_{\leq n})\big|= f_n$, for every
  $n<\omega$. Then $f$ is continuous and clearly has the property in
  \eqref{eq:Select-Ext-v14:1}.
\end{proof}

\section{Selections and canonical maps}
\label{sec:select-canon-maps}

Suppose that $X$ is a (paracompact) space with the property that for
any space $Y$, each aspherical sequence $\varphi_n:X\sto Y$,
$n<\omega$, of lower locally constant mappings admits a continuous
selection for its union $\bigcup_{n<\omega}\varphi_n$. As we will see
in the next section (Corollaries \ref{corollary-Select-Ext-v18:1},
\ref{corollary-Select-Ext-v18:2} and \ref{corollary-Select-Ext-v18:3}
and Example \ref{example-Select-Ext-v18:1}), each one of the following
statements determines a different dimension-like property of
$X$.\smallskip

\begin{enumerate}[label=\upshape{(\thesection.\arabic*)}]
\item\label{item:Select-Ext-v18:1} There exists an aspherical sequence
  $\varphi_n:X\sto Y$, $n<\omega$, of lower locally constant mappings
  such that no $\varphi_n$, $n<\omega$, has a continuous
  selection.\smallskip
\item\label{item:Select-Ext-v18:2} For each aspherical sequence
  $\varphi_k:X\sto Y$, $k<\omega$, of lower locally constant mappings
  there exists an $n<\omega$ such that $\varphi_n$ has a continuous
  selection.\smallskip
\item\label{item:Select-Ext-v18:3} There exists an $n<\omega$ such
  that for each aspherical sequence $\varphi_k:X\sto Y$, $k<\omega$,
  of lower locally constant mappings, the mapping $\varphi_n$ has a
  continuous selection.
\end{enumerate}\smallskip

Here, we deal with the following general result reducing these
selection problems only to simpli\-cial-valued mappings associated to
open locally finite covers of $X$.

\begin{theorem}
  \label{theorem-Select-Ext-vgg-rev:1}
  For a space $Y$, a paracompact space $X$ and $0<\mu\leq \omega+1$,
  the following are equivalent\textup{:}
  \begin{enumerate}
  \item\label{item:Select-Ext-v14:1} If $\varphi_n:X\sto Y$,
    $n<\omega$, is an aspherical sequence of lower locally constant
    mappings, then $\bigcup_{n<\kappa}\varphi_n$ has a
    continuous selection for some $0<\kappa<\mu$.
  \item\label{item:Select-Ext-v14:2} If $\mathscr{U}_n$, $n<\omega$,
    is a sequence of open locally finite covers of $X$, then
    ${\left|\Delta_{[\mathscr{U}_{<\kappa}]}\right|:X\sto
      \big|\Delta(\mathscr{U}_{<\kappa}) \big|}$ has a continuous
    selection for some $0<\kappa<\mu$.
  \end{enumerate}
\end{theorem}

The proof of Theorem \ref{theorem-Select-Ext-vgg-rev:1} is based on
the results of the previous two sections and the following
observation.

\begin{proposition}
  \label{proposition-Select-Ext-v21:1}
  Let $\mathscr{U}_n$, $n<\omega$, be a sequence of covers of $X$,
  $0<\kappa\leq \omega$, and $\mathscr{V}_n$, $n<\kappa$, be a sequence of
  families of subsets of $X$ such that each $\mathscr{V}_n$ refines
  $\mathscr{U}_n$ and $\bigcup_{n<\kappa}\mathscr{V}_n$ is a cover of
  $X$. If\/ $\left|\Delta_{[\mathscr{V}_{<\kappa}]}\right|:X\sto
  \left|\Delta(\mathscr{V}_{<\kappa})\right|$ has a continuous selection,
  then so does $\left|\Delta_{[\mathscr{U}_{<\kappa}]}\right|:X\sto
  \left|\Delta(\mathscr{U}_{<\kappa})\right|$.
\end{proposition}

\begin{proof}
  Since each $\mathscr{V}_n$ refines $\mathscr{U}_n$, there are maps
  $r_n:\mathscr{V}_n\to \mathscr{U}_n$, $n<\kappa$, such that
  $V\subset r_n(V)$, for all $V\in \mathscr{V}_n$. Accordingly,
  $r=\bigsqcup_{n<\kappa}r_n:\Delta(\mathscr{V}_{<\kappa})\to
  \Delta(\mathscr{U}_{<\kappa})$ is a simplicial map with the property
  that $\sigma\subset r(\sigma)$, for each simplex
  $\sigma\in \Delta(\mathscr{V}_{<\kappa})$. In other words,
  $r\circ \Delta_{[\mathscr{V}_{<\kappa}]}\subset
  \Delta_{[\mathscr{U}_{<\kappa}]}$, see \eqref{eq:Select-Ext-v15:3}, and
  therefore
  $|r|\circ \left|\Delta_{[\mathscr{V}_{<\kappa}]}\right|\subset
  \left|\Delta_{[\mathscr{U}_{<\kappa}]}\right|$. Thus, if
  $h:X\to \left|\Delta(\mathscr{V}_{<\kappa})\right|$ is a continuous
  selection for
  $\left|\Delta_{[\mathscr{V}_{<\kappa}]}\right|:X\sto
  \left|\Delta(\mathscr{V}_{<\kappa})\right|$, then the composite map
  $f= |r|\circ h:X\to \left|\Delta(\mathscr{U}_{<\kappa})\right|$ is a
  continuous selection for
  $\left|\Delta_{[\mathscr{U}_{<\kappa}]}\right|:X\sto
  \left|\Delta(\mathscr{U}_{<\kappa})\right|$.
\end{proof}

\begin{proof}[Proof of Theorem \ref{theorem-Select-Ext-vgg-rev:1}] 
  The implication
  \ref{item:Select-Ext-v14:1}$\implies$\ref{item:Select-Ext-v14:2}
  follows from Propositions \ref{proposition-Select-Ext-v16:1} and
  \ref{proposition-Select-Ext-v16:2}. The converse follows easily from
  Theorem \ref{theorem-Select-Ext-v14:1} and Proposition
  \ref{proposition-Select-Ext-v21:1}. Namely, assume that
  \ref{item:Select-Ext-v14:2} holds and $\varphi_n:X\sto Y$,
  $n<\omega$, is as in \ref{item:Select-Ext-v14:1}. Since $X$ is
  paracompact, by Theorem \ref{theorem-Select-Ext-v14:1}, there exists
  a sequence $\mathscr{F}_n$, $n<\omega$, of closed locally finite
  interior covers of $X$ and a continuous map
  $f:\big|\Delta_{[\mathscr{F}_{<\omega}]}\big|\to Y$ satisfying
  \eqref{eq:Select-Ext-v14:1}. For each $n<\omega$, let
  $\mathscr{U}_n$ be the cover of $X$ composed by the interiors of the
  elements of $\mathscr{F}_n$. Then by \ref{item:Select-Ext-v14:2},
  the mapping
  $\left|\Delta_{[\mathscr{U}_{<\kappa}]}\right|:X\sto
  \big|\Delta(\mathscr{U}_{<\kappa}) \big|$ has a continuous selection
  for some $0<\kappa<\mu$. According to Proposition
  \ref{proposition-Select-Ext-v21:1}, this implies that the mapping
  $\left|\Delta_{[\mathscr{F}_{<\kappa}]}\right|:X\sto
  \left|\Delta(\mathscr{F}_{<\kappa})\right|$ also has a continuous
  selection $h:X\to \big|\Delta(\mathscr{F}_{<\kappa}) \big|$. Evidently,
  the composite map $g=f\circ h:X\to Y$ is a continuous selection for
  the mapping $\bigcup_{n<\kappa}\varphi_n$.
\end{proof}

The selection problem in \ref{item:Select-Ext-v14:2} of Theorem
\ref{theorem-Select-Ext-vgg-rev:1} is naturally related to the
existence of canonical maps for the disjoint union
$\bigsqcup_{n<\kappa}\mathscr{U}_n$ of such covers.  To this end, let us
briefly recall some terminology.  For a simplicial complex $\Sigma$
and a simplex $\sigma\in \Sigma$, we use $\langle\sigma\rangle$ to
denote the \emph{relative interior} of the geometric simplex
$|\sigma|$. For a vertex $v\in \Sigma^0$, the set
\setcounter{equation}{3}
\begin{equation}
  \label{eq:Select-Ext-v7:4}
  \st\langle v\rangle
  =\bigcup_{v\in \sigma\in \Sigma}\langle\sigma\rangle,
\end{equation}
is called the \emph{open star} of the vertex $v\in \Sigma^0$. One can
easily see that $\st\langle v\rangle$ is open in $|\Sigma|$ because
$\st\langle v\rangle=|\Sigma|\setminus\bigcup_{v\notin \sigma\in
  \Sigma}|\sigma|$. In these terms, for an indexed cover
$\{U_\alpha:\alpha\in \mathscr{A}\}$ of a space $X$, a continuous map
$f:X\to |\mathscr{N}(\mathscr{A})|$ is called \emph{canonical} for
  $\{U_\alpha:\alpha\in \mathscr{A}\}$ if
\begin{equation}
  \label{eq:Select-Ext-v7:1}
  f^{-1}(\st\langle \alpha\rangle)\subset U_\alpha,\quad \text{for
    every $\alpha\in 
    \mathscr{A}$.} 
\end{equation}      

It is well known that each open cover of a paracompact space admits a
canonical map, which follows from the fact that such a cover has an
index-subordinated partition of unity. The interested reader is
referred to \cite[Section 2]{gutev:2018a} which contains a brief
review of several facts about canonical maps and partitions of unity.
Here, we are interested in a selection interpretation of canonical
maps.  Namely, in terms of the simplicial-valued mapping
$\Sigma_\mathscr{A}:X\sto \mathscr{N}(\mathscr{A})$ associated to the
cover $\{U_\alpha:\alpha\in \mathscr{A}\}$, see
\eqref{eq:Select-Ext-v15:2}, we have the following characterisation of
canonical maps; for unindexed covers it was obtained in
\cite[Proposition 2.5]{gutev:2018a} (see also Dowker
\cite{dowker:47}), but the proof for indexed covers is essentially the
same.

\begin{proposition}
  \label{proposition-Select-Ext-v11:1}
  A map $f:X\to |\mathscr{N}(\mathscr{A})|$ is canonical for a cover
  $\{U_\alpha:\alpha\in \mathscr{A}\}$ of a space $X$ if and only if
  it is a continuous selection for the associated mapping
  $|\Sigma_\mathscr{A}|:X\sto |\mathscr{N}(\mathscr{A})|$.
\end{proposition}       

In the special case of a sequence of open covers $\mathscr{U}_n$,
$n<\omega$, a canonical map
$f:X\to \mathscr{N}(\mathscr{U}_{<\omega})$ for the disjoint union
$\bigsqcup_{n<\omega}\mathscr{U}_n$ will be called \emph{canonical}
for the sequence $\mathscr{U}_n$, $n<\omega$.  We now have the
following further reduction of the selection problem for aspherical
sequences of mappings, which is an immediate consequence of Theorem
\ref{theorem-Select-Ext-vgg-rev:1} and Proposition
\ref{proposition-Select-Ext-v11:1}. 

\begin{corollary}
  \label{corollary-Select-Ext-v22:1}
  For a space $Y$, a paracompact space $X$ and $0<\mu\leq \omega+1$,
  the following are equivalent\textup{:}
  \begin{enumerate}
  \item\label{item:Select-Ext-v18:4} If $\varphi_n:X\sto Y$,
    $n<\omega$, is an aspherical sequence of lower locally constant
    mappings, then $\bigcup_{n<\kappa}\varphi_n$ has a
    continuous selection for some $0<\kappa<\mu$.
  \item\label{item:Select-Ext-v18:5} Each sequence $\mathscr{U}_n$,
    $n<\omega$, of open covers of $X$ admits a canonical map
    $f:X\to \big|\Delta(\mathscr{U}_{<\kappa}) \big|\subset
    \left|\mathscr{N}(\mathscr{U}_{<\omega})\right|$ for some
    $0<\kappa<\mu$.
  \end{enumerate}
\end{corollary}
  
\section{Dimension and canonical maps}
\label{sec:canon-maps-sequ}

Here, we finalise the proof of Theorem \ref{theorem-Select-Ext-v6:2}
by showing that the property $C$ is equivalent to the existence of
canonical maps for special covers. To this end, for a sequence
$\mathscr{U}_n$, $n<\omega$, of open covers of $X$ and
$0<\kappa\leq \omega$, we shall say that a sequence $\mathscr{V}_n$,
$n<\kappa$, of pairwise-disjoint families of open subsets $X$ is a
\emph{$C$-refinement} of $\mathscr{U}_n$, $n<\omega$, if each family
$\mathscr{V}_n$ refines $\mathscr{U}_n$ and
$\bigcup_{n<\kappa}\mathscr{V}_n$ covers $X$.

\begin{theorem}
  \label{theorem-Select-Ext-v18:2}
  For a paracompact space $X$ and $0<\mu\leq \omega+1$, the following
  are equivalent\textup{:}
  \begin{enumerate}
  \item\label{item:Select-Ext-v18:6} Each sequence $\mathscr{U}_n$,
    $n<\omega$, of open covers of $X$ has a $C$-refinement
    $\mathscr{V}_n$, $n<\kappa$, for some $0<\kappa<\mu$.
  \item\label{item:Select-Ext-v18:7} Each sequence $\mathscr{U}_n$,
    $n<\omega$, of open covers of $X$ admits a canonical map
    $f:X\to \big|\Delta(\mathscr{U}_{<\kappa}) \big|$ for some
    $0<\kappa<\mu$.
  \end{enumerate}
\end{theorem}

\begin{proof}
  To see that
  \ref{item:Select-Ext-v18:6}$\implies$\ref{item:Select-Ext-v18:7},
  take a sequence $\mathscr{U}_n$, $n<\omega$, of open covers of
  $X$. Then by \ref{item:Select-Ext-v18:6}, $\mathscr{U}_n$,
  $n<\omega$, admits a $C$-refinement $\mathscr{V}_n$, $n<\kappa$, for
  some $0<\kappa<\mu$.  Let $\mathscr{N}(\mathscr{V}_{<\kappa})$ be
  the nerve of the disjoint union $\bigsqcup_{n<\kappa}\mathscr{V}_n$,
  see \eqref{eq:Select-Ext-v12:2} of
  Example~\ref{example-Select-Ext-v12:1}, and
  $\Sigma_{[\mathscr{V}_{<\kappa}]}:X\sto
  \mathscr{N}(\mathscr{V}_{<\kappa})$ be the simplicial-valued mapping
  associated to this nerve, see \eqref{eq:Select-Ext-v15:2}.  Since
  $X$ is paracompact, the indexed cover
  $\bigsqcup_{n<\kappa}\mathscr{V}_n$ has a canonical map.  Hence, by
  Proposition \ref{proposition-Select-Ext-v11:1}, the mapping
  $\left|\Sigma_{[\mathscr{V}_{<\kappa}]}\right|:X\sto
  \left|\mathscr{N}(\mathscr{V}_{<\kappa})\right|$ has a continuous
  selection. However, by definition, each family $\mathscr{V}_n$,
  $n<\kappa$, is pairwise-disjoint. Therefore, by
  \eqref{eq:Select-Ext-v12:3} of Proposition
  \ref{proposition-Select-Ext-v16:3},
  $\Delta(\mathscr{V}_{<\kappa})=\mathscr{N}(\mathscr{V}_{<\kappa})$
  and, consequently,
  $\Delta_{[\mathscr{V}_{<\kappa}]}=\Sigma_{[\mathscr{V}_{<\kappa}]}$. Thus,
  $\left|\Delta_{[\mathscr{V}_{<\kappa}]}\right|:X\sto
  \left|\Delta(\mathscr{V}_{<\kappa})\right|$ has a continuous
  selection and, according to Proposition
  \ref{proposition-Select-Ext-v21:1}, the mapping
  $\left|\Delta_{[\mathscr{U}_{<\kappa}]}\right|:X\sto
  \left|\Delta(\mathscr{U}_{<\kappa})\right|$ has a continuous
  selection as well.  Finally, by Proposition
  \ref{proposition-Select-Ext-v11:1}, each continuous selection for
  $\left|\Delta_{[\mathscr{U}_{<\kappa}]}\right|$ is as required in
  \ref{item:Select-Ext-v18:7}. \medskip
  
  Conversely, let $\mathscr{U}_n$, $n<\omega$, and
  ${f:X\to \left|\Delta(\mathscr{U}_{<\kappa})\right|}$ be as in
  \ref{item:Select-Ext-v18:7} for some ${0<\kappa<\mu}$.  Define
  $\mathscr{V}_n=\left\{f^{-1}(\st\langle U\rangle): U\in
    \mathscr{U}_n\right\}$, $n<\kappa$. Since $f$ is continuous,
  $\mathscr{V}_n$ is an open family in $X$; moreover, by
  \eqref{eq:Select-Ext-v7:1}, it refines $\mathscr{U}_n$. It is also
  evident that $\bigcup_{n<\mu}\mathscr{V}_n$ covers $X$, see
  \eqref{eq:Select-Ext-v7:4}. We complete the proof by showing that
  $\mathscr{V}_n$ is pairwise-disjoint as well. To this end, suppose
  that
  $p\in f^{-1}(\st\langle U_1\rangle)\cap f^{-1}(\st\langle
  U_2\rangle)$ for some $U_1,U_2\in \mathscr{U}_n$ and $p\in X$.  Then
  $f(p)\in \st\langle U_1\rangle\cap \st\langle U_2\rangle$ and by
  \eqref{eq:Select-Ext-v7:4}, we have that
  $f(p)\in \langle\sigma_1\rangle\cap \langle\sigma_2\rangle$ for some
  simplices $\sigma_1,\sigma_2\in \Delta(\mathscr{U}_{<\kappa})$ with
  $U_i\in \sigma_i$, $i=1,2$. Since the collection
  $\left\{\langle \sigma\rangle: \sigma\in
    \Delta(\mathscr{U}_{<\kappa})\right\}$ forms a partition of
  $\left|\Delta(\mathscr{U}_{<\kappa})\right|$, this implies that
  $\sigma_1=\sigma_2$. Finally, according to the definition of
  $\Delta(\mathscr{U}_{<\mu})$, see \eqref{eq:Select-Ext-v12:4}, we
  get that $U_1=U_2$. Thus, each family $\mathscr{V}_n$, $n<\mu$, is
  also pairwise-disjoint, and the proof is complete.
\end{proof}

We finalise the paper with several applications. The first one is the
following slight generalisation of Theorem
\ref{theorem-Select-Ext-v6:2}; it is an immediate consequence of
Corollary~\ref{corollary-Select-Ext-v22:1} and Theorem
\ref{theorem-Select-Ext-v18:2} (in the special case of
$\mu=\omega+1$).

\begin{corollary}
  \label{corollary-Select-Ext-v18:1}
  For a paracompact space $X$, the following are equivalent\textup{:}
  \begin{enumerate}
  \item\label{item:Select-Ext-v7:1} $X$ is a $C$-space.
  \item\label{item:Select-Ext-v7:2} For every space $Y$, each
    aspherical sequence $\varphi_n : X\sto Y$, $n<\omega$, of lower
    locally constant mappings admits a continuous selection for its
    union $\bigcup_{n<\omega} \varphi_n$.
  \item\label{item:Select-Ext-v7:3} Each sequence $\mathscr{U}_n$,
    $n<\omega$, of open covers of $X$ admits a canonical map
    $f:X\to \big|\Delta(\mathscr{U}_{<\omega})\big|$.
  \end{enumerate}
\end{corollary}

Another consequence is for the case when $0<\mu=\omega$, and deals
with the so called finite $C$-spaces. These spaces were defined by
Borst for separable metrizable spaces, see \cite{MR2280911};
subsequently, the definition was extended by Valov \cite{valov:00} for
arbitrary spaces. For simplicity, we will consider these spaces in the
realm of normal spaces. In this setting, a (normal) space $X$ is
called a \emph{finite $C$-space} if for any sequence
$\{\mathscr{U} _k:k<\omega\}$ of finite open covers of $X$ there
exists a finite sequence $\{\mathscr{V} _k:k\leq n\}$ of open
pairwise-disjoint families in $X$ such that each $\mathscr{V} _k$
refines $\mathscr{U} _k$ and $\bigcup_{k\leq n}\mathscr{V} _k$ is a
cover of $X$.  It was shown by Valov in \cite[Theorem 2.4]{valov:00}
that a paracompact space $X$ is a finite $C$-space if and only if each
sequence $\{\mathscr{U} _k:k<\omega\}$ of open covers of $X$ admits a
finite $C$-refinement, i.e.\ there exists a finite sequence
$\{\mathscr{V} _k:k\leq n\}$ of open pairwise-disjoint families in $X$
such that each $\mathscr{V} _k$ refines $\mathscr{U} _k$ and
$\bigcup_{k\leq n}\mathscr{V} _k$ is a cover of $X$. Based on this, we
have the following consequence of Corollary
\ref{corollary-Select-Ext-v22:1} and Theorem
\ref{theorem-Select-Ext-v18:2} (in the special case of $\mu=\omega$).

\begin{corollary}
  \label{corollary-Select-Ext-v18:2}
  For a paracompact space $X$, the following are equivalent\textup{:}
  \begin{enumerate}
  \item\label{item:Select-Ext-v12:1} $X$ is a finite $C$-space.
  \item\label{item:Select-Ext-v12:2} For each aspherical sequence
    $\varphi_k : X\sto Y$, $k<\omega$, of lower locally constant
    mappings in a space $Y$, there exists $n<\omega$ such that
    $\varphi_n$ has a continuous selection.
  \item\label{item:Select-Ext-v12:3} Each sequence $\mathscr{U}_k$,
    $k<\omega$, of open covers of $X$ admits a canonical map
    $f:X\to \big|\Delta(\mathscr{U}_{\leq n})\big|$ for some
    $n<\omega$.
  \end{enumerate}
\end{corollary}

Let us explicitly remark that the equivalence
\ref{item:Select-Ext-v12:1}$\iff$\ref{item:Select-Ext-v12:2} in
Corollary \ref{corollary-Select-Ext-v18:2} was obtained by Valov in
\cite[Theorem 1.1]{valov:00}.  His arguments were following those in
\cite{uspenskij:98} for proving Theorem \ref{theorem-Select-Ext-v6:1}.
Accordingly, our approach is providing a simplification of this
proof. Regarding the proper place of finite $C$-spaces, it was shown
by Valov in \cite[Proposition 2.2]{valov:00} that a Tychonoff space
$X$ is a finite $C$-space if and only if its \v{C}ech-Stone
compactification $\beta X$ is a $C$-space. This brings a natural
distinction between the selection problems stated in
\ref{item:Select-Ext-v18:1} and \ref{item:Select-Ext-v18:2}.

\begin{example}
  \label{example-Select-Ext-v18:1}
  The following example of a $C$-space which is not finite $C$ was
  given in \cite[Remark 3.7]{MR2080284}. Let $K_\omega$ be the
  subspace of the Hilbert cube $[0,1]^\omega$ consisting of all points
  which have only finitely many nonzero coordinates. Then $K_\omega$
  is a $C$-space being strongly countable-dimensional, but is not a
  finite $C$-space because each compactification of $K_\omega$
  contains a copy of $[0,1]^\omega$ (as per \cite[Example
  5.5.(1)]{MR722011}). According to Corollaries
  \ref{corollary-Select-Ext-v18:1} and
  \ref{corollary-Select-Ext-v18:2}, see also Theorem
  \ref{theorem-Select-Ext-vgg-rev:1}, this implies that there exists a
  space $Y$ and an aspherical sequence $\varphi_n:K_\omega\sto Y$,
  $n<\omega$, of lower locally constant mappings such that
  $\bigcup_{n<\omega}\varphi_n$ has a continuous selection, but none
  of the mappings $\varphi_n$, $n<\omega$, has a continuous
  selection.\qed
\end{example}

Our last application is for the case when $\mu=n+1$ for some
$n<\omega$. To this end, following \cite{gutev:2018a}, a finite
sequence $\varphi_k:X\sto Y$, $0\leq k\leq n$, of mappings will be
called \emph{aspherical} if $\varphi_k\embed{k} \varphi_{k+1}$, for
every $k<n$. By letting $\varphi_k(p)=Y*q$ be the cone on $Y$ with a
fixed vertex $q$, where $p\in X$ and $k>n$, each finite aspherical
sequence $\varphi_k:X\sto Y$, $0\leq k\leq n$, can be extended to an
aspherical sequence $\varphi_k:X\sto Y*q$, $k<\omega$. Furthermore, in
this construction, each resulting new mapping $\varphi_k$, $k>n$, is
lower locally constant being a constant set-valued mapping.\medskip

Regarding dimension properties of the domain, let us recall a result
of Ostrand \cite{ostrand:71} that for a normal space $X$ with a
covering dimension $\dim(X)\leq n$, each open locally finite cover
$\mathscr{U}$ of $X$ admits a sequence
$\mathscr{V}_0, \dots, \mathscr{V}_n$ of open pairwise-disjoint
families such that each $\mathscr{V}_k$ refines $\mathscr{U}_k$ and
$\bigcup_{k=0}^n\mathscr{V}_k$ covers $X$. This result was refined by
Addis and Gresham, see \cite[Proposition 2.12]{addis-gresham:78}, that
a paracompact space $X$ has a covering dimension $\dim(X)\leq n$ if
and only if each finite sequence $\mathscr{U}_0,\dots, \mathscr{U}_n$
of open covers of $X$ has a finite $C$-refinement, i.e.\ there exists
a finite sequence $\mathscr{V}_0,\dots, \mathscr{V}_n$ of open
pairwise-disjoint families of $X$ such that each $\mathscr{V}_k$
refines $\mathscr{U}_k$ and $\bigcup_{k=0}^n \mathscr{V}_k$ covers
$X$. Just like before, setting $\mathscr{U}_k=\mathscr{U}_n$, $k>n$,
the above characterisation of the covering dimension of paracompact
spaces remains valid for an infinite sequence $\mathscr{U}_k$,
$k<\omega$, of open covers of $X$. Accordingly, we also have the
following consequence of Corollary \ref{corollary-Select-Ext-v22:1}
and Theorem \ref{theorem-Select-Ext-v18:2} (in the special case of
$\mu=n+1<\omega$).

\begin{corollary}
  \label{corollary-Select-Ext-v18:3}
    For a paracompact space $X$, the following are equivalent\textup{:}
    \begin{enumerate}
    \item\label{item:Select-Ext-v18:8} $\dim(X)\leq n$.
    \item\label{item:Select-Ext-v18:9} For each aspherical sequence
      $\varphi_k : X\sto Y$, $0\leq k\leq n$, of lower locally
      constant mappings in a space $Y$, the mapping $\varphi_n$ has
      a continuous selection.
    \item\label{item:Select-Ext-v18:10} Each sequence $\mathscr{U}_k$,
      $0\leq k\leq n$, of open covers of $X$ admits a canonical map
      $f:X\to \big|\Delta(\mathscr{U}_{\leq n})\big|$.
    \end{enumerate}
\end{corollary}

A direct poof of the implication
\ref{item:Select-Ext-v18:8}$\implies$\ref{item:Select-Ext-v18:9} in
Corollary \ref{corollary-Select-Ext-v18:3} was given in \cite[Theorem
3.1]{gutev:2018a}. Let us also remark that in the special case when
all mappings $\varphi_k$, $0\leq k\leq n$, are equal, the equivalence
of \ref{item:Select-Ext-v18:8} and \ref{item:Select-Ext-v18:9} in
Corollary \ref{corollary-Select-Ext-v18:3} was shown in \cite[Remark
2]{uspenskij:98} and credited to Ernest Michael.


\begin{thebibliography}{10}

\bibitem{addis-gresham:78}
D.~Addis and J.~Gresham, \emph{A class of infinite-dimensional spaces. {I}.
  {D}imension theory and {A}lexandroff's problem}, Fund. Math. \textbf{101}
  (1978), no.~3, 195--205.

\bibitem{MR2280911}
P.~Borst, \emph{Some remarks concerning {C}-spaces}, Topology Appl.
  \textbf{154} (2007), no.~3, 665--674.

\bibitem{chigogidze-valov:00a}
A.~Chigogidze and V.~Valov, \emph{Extensional dimension and {$C$}-spaces},
  Bull. London Math. Soc. \textbf{34} (2002), no.~6, 708--716.

\bibitem{dowker:47}
C.~H. Dowker, \emph{Mappings theorems for non-compact spaces}, Amer. J. Math.
  \textbf{69} (1947), 200--242.

\bibitem{MR722011} R.~Engelking and E.~Pol,
  \emph{Countable-dimensional spaces: a survey}, Dissertationes
  Math. (Rozprawy Mat.) \textbf{216} (1983), 41 pp.

\bibitem{MR2352366}
V.~V. Fedorchuk, \emph{Weakly infinite-dimensional spaces}, Uspekhi
  Mat. Nauk \textbf{62} (2007), no.~2(374), 109--164; translation in
Russian Math. Surveys \textbf{62} (2007), no.~2, 323--374.

\bibitem{gutev:95e}
V.~Gutev, \emph{Factorizations of set-valued mappings with separable range},
  Comment. Math. Univ. Carolin. \textbf{37} (1996), no.~4, 809--814.

\bibitem{gutev:05}
\bysame, \emph{Selections and approximations in finite-dimensional spaces},
  Topology Appl. \textbf{146--147} (2005), 353--383.

\bibitem{gutev:2018a}
\bysame, \emph{Constructing selections stepwise over skeletons of nerves of
  covers}, Serdica Math. J. \textbf{44} (2018), 137--154.

\bibitem{haver:1974}
W.~Haver, \emph{A covering property for metric spaces}, Topology {C}onference
  ({V}irginia {P}olytech. {I}nst. and {S}tate {U}niv., {B}lacksburg, {V}a.,
  1973), Springer, Berlin, 1974, pp.~108--113. Lecture Notes in Math., Vol.
  375.

\bibitem{MR2080284}
T.~Kimura and C.~Komoda, \emph{Spaces having a compactification which is a
  {$C$}-space}, Topology Appl. \textbf{143} (2004), no.~1-3, 87--92.

\bibitem{MR0007093}
S.~Lefschetz, \emph{Algebraic {T}opology}, American Mathematical Society
  Colloquium Publications, v. 27, American Mathematical Society, New York,
  1942.

\bibitem{ostrand:71}
P.~Ostrand, \emph{Covering dimension in general spaces}, Topology Appl.
  \textbf{1} (1971), 209--221.

\bibitem{pol:81}
R.~Pol, \emph{A weakly infinite-dimensional compactum which is not
  countable-dimensional}, Proc. Amer. Math. Soc. \textbf{82} (1981), no.~4,
  634--636.

\bibitem{uspenskij:98}
V.~Uspenskij, \emph{A selection theorem for {$C$}-spaces}, Topology Appl.
  \textbf{85} (1998), 351--374.

\bibitem{valov:00}
V.~Valov, \emph{Continuous selections and finite {$C$}-spaces}, Set-Valued
  Anal. \textbf{10} (2002), no.~1, 37--51.

\bibitem{MR1576810} J.~H.~C. Whitehead, \emph{Simplicial {S}paces,
    {N}uclei and m-{G}roups}, Proc. London Math. Soc. (2) \textbf{45}
  (1939), no.~4, 243--327.

\bibitem{MR0030759}
\bysame, \emph{Combinatorial homotopy. {I}}, Bull.  Amer.
  Math. Soc. \textbf{55} (1949), 213--245.

\end{thebibliography}

\providecommand{\bysame}{\leavevmode\hbox to3em{\hrulefill}\thinspace}

\end{document}